\documentclass[12pt]{amsart}
\usepackage{amssymb}
\usepackage{hyperref,hyphenat}
\usepackage[all]{xy}
\usepackage[toc,page,title,titletoc,header]{appendix}
\usepackage{xcolor}

\numberwithin{equation}{section}

\theoremstyle{plain}
\newtheorem{Thm}{Theorem}[section]
\newtheorem{Prop}[Thm]{Proposition}
\newtheorem{Lem}[Thm]{Lemma}

\newtheorem{Cor}[Thm]{Corollary}
\theoremstyle{remark}
\newtheorem{Rem}[Thm]{Remark}
\theoremstyle{definition}
\newtheorem{Def}[Thm]{Definition}

\def\A{\mathbb{A}}
\def\Z{\mathbb{Z}}
\def\Q{\mathbb{Q}}
\def\R{\mathbb{R}}
\def\C{\mathbb{C}}
\def\P{\mathbb{P}}

\def\sL{\mathcal{L}}
\def\sT{\mathcal{T}}
\def\sD{\mathcal{D}}

\def\Lyap{\mathrm{Lyap}}
\def\Aff{\mathrm{Aff}}
\def\Spec{\mathrm{Spec}}
\def\PrePer{\mathrm{PrePer}}
\def\Poly{\mathrm{Poly}}
\def\MPoly{\mathrm{MPoly}}
\def\Zar{\mathrm{Zar}}
\def\Gal{\mathrm{Gal}}
\def\crit{\mathrm{crit}}
\def\Man{\mathrm{Man}}

\def\la{\lambda}

\begin{document}
\title[]{Rigidity of Lyapunov exponents for polynomials}

\author{Zhuchao Ji}
\address{Institute for Theoretical Sciences, Westlake University, Hangzhou 310030, China}
\email{jizhuchao@westlake.edu.cn}

\author{Junyi Xie}
\address{Beijing International Center for Mathematical Research, Peking University, Beijing 100871, China}
\email{xiejunyi@bicmr.pku.edu.cn}

\author{Geng-Rui Zhang}
\address{School of Mathematical Sciences, Peking University, Beijing 100871, China}
\email{grzhang@stu.pku.edu.cn, chibasei@163.com}



\begin{abstract}
	Let $f,g\in\overline{\Q}[z]$ be polynomials of degree $d\geq2$ with disconnected Julia sets. We prove that they have the same Lyapunov exponent $\sL_f=\sL_g$ if and only if either $f$ and $g$ are intertwined, or $f$ and $\overline{g}$ are intertwined. The analogous result for critical heights is also obtained. As an application, we provide a new proof of the theorem stating that the multiplier spectrum morphism on the moduli space of polynomials is generically injective.
\end{abstract}

\maketitle

\section{Introduction}

\subsection{Statement of the main theorem}
Let $f\in\C[z]$ be a polynomial of degree $d\geq2$. The \emph{unique maximal entropy measure} of $f$ on $\P^1(\C)$ is denoted by $\mu_f$; see \cite{Lyubich83,FLM83,Mane83}. The measure $\mu_f$ is supported on the \emph{Julia set} $J(f)\subset\C$ of $f$, which is compact in $\C$. The \emph{Lyapunov exponent} (with respect to $\mu_f$) of $f$ is defined by
$$\sL_f:=\int_{\C}\log|f^\prime|\ d\mu_f.$$
It is well-known that $\sL_f\geq\log(d)$, and equality holds if and only if $J(f)$ is connected (cf. \cite{Prz85}). The quantity $\exp(\sL_f)$ is considered as an average of $|f^\prime|$ over $J(f)$, which measures the average expansion rate of $f$ along a typical orbit.

We consider the following algebraic relation between polynomials:
\begin{Def}\label{intwdef}
	Let $f,g\in\C[z]$ be polynomials of degree $\geq2$. We say that $f$ and $g$ are \emph{intertwined} if there exists a (possibly reducible) algebraic curve $Z\subset\A^2$ over $\C$ whose projections to both axes are surjective such that $Z$ is invariant under the endomorphism $(f,g):\A^2\to\A^2,(x,y)\mapsto(f(x),g(y))$.
\end{Def}
The polynomials $f$ and $g$ are intertwined if and only if there exist non-constant polynomials $R,h_1,h_2\in\C[z]\setminus\C$ and $n\in\Z_{>0}$ such that
\begin{equation}\label{intwdef2}
	f^{\circ n}\circ h_1=h_1\circ R\quad\text{and}\quad g^{\circ n}\circ h_2=h_2\circ R.
\end{equation}
(Here $f^{\circ n}$ denotes the $n$-th iterate of $f$.) See \cite[Theorem~3.39]{FG22}. In particular, $\deg(f)=\deg(g)$ if $f$ and $g$ are intertwined.

Let $\overline{f}\in\C[z]$ be the polynomial of degree $d$ obtained from $f$ by applying complex conjugation to all coefficients. Then $\sL_f=\sL_{\overline{f}}$ (see Lemma~\ref{intwLyapsame}).

From now on, we fix an algebraic closure $\overline{\Q}$ of $\Q$ in $\C$ and view all number fields as subfields of $\C$. All polynomials with coefficients in $\overline{\Q}$ are also viewed as polynomials with coefficients in $\C$ by the fixed embedding $\overline{\Q}\hookrightarrow\C$. Our aim is to study the rigidity of Lyapunov exponents for polynomials. The main theorem of this paper asserts that for polynomials defined over $\overline{\Q}$ with disconnected Julia sets, Lyapunov exponents determine their intertwining classes:
\begin{Thm}\label{thmmain}
	Let $f,g\in\overline{\Q}[z]$ be polynomials of degree $d\geq2$. Assume that $\sL_f>\log(d)$ (equivalently, $J(f)$ is disconnected). Then $\sL_f=\sL_g$ if and only if $f$ and $\hat{g}$ are intertwined for some $\hat{g}\in\{g,\overline{g}\}$.
\end{Thm}
\begin{Rem}\label{rmkcond}
	Let $d\geq2$ be an integer. By considering the unicritical family $\{z^d+c\}_{c\in\C}$, it is easy to see that for every real number $L_0\geq\log(d)$, there is an uncountable subset $A_{d,L_0}$ of $\C[z]$ consisting of polynomials of degree $d$ with Lyapunov exponent $L_0$ such that for all distinct $f,g\in A_{d,L_0}$, $f$ and $g$ are not conjugate. Combined with the result \cite[Theorem~3.46]{FG22} on the finiteness of intertwining classes (see also \cite{Pakintw}), we conclude that for every real number $L_0>\log(d)$, there exist polynomials $f,g\in\C[z]$ of degree $d$ satisfying:
	\begin{itemize}
		\item $\sL_f=\sL_g=L_0$;
		\item $f$ and $g$ are not intertwined;
		\item $f$ and $\overline{g}$ are not intertwined.
	\end{itemize}
	(In fact, the set of such polynomials can be chosen to be uncountable.) Thus, we cannot remove the condition that $f$ and $g$ have $\overline{\Q}$-coefficients in Theorem~\ref{thmmain}.
	
	It is also easy to see that there exist $f,g\in\overline{\Q}[z]$ of degree $d$ with
	$$\sL_f=\sL_g=\log(d),$$
	but $f$ and $\hat{g}$ are not intertwined for any $\hat{g}\in\{g,\overline{g}\}$. So the condition that $\sL_f>\log(d)$ cannot be removed either.
\end{Rem}
\begin{Rem}
	Let $d\geq2$ be an integer. By \cite[Theorems~3.51~and~3.52]{FG22}, a ``general'' polynomial $f$ of degree $d$ satisfies the property that every polynomial $g$ of degree $d$ which is intertwined with $f$ must be conjugate to $f$. Here ``general'' means the property holds for $f$ in a Zariski open dense subset of the parameter space $\Poly^d$.
\end{Rem}
The proof of Theorem~\ref{thmmain} will be given in \S~\ref{secpf}, based on the second-named author's result \cite[Theorem~1.13]{trans} on the transcendence of products of B{\"o}ttcher coordinates (see Theorem~\ref{transBott}). In the process of argument, we can prove that for every polynomial $f\in\overline{\Q}[z]$ of degree $d\geq2$, either $\sL_f=\log(d)$ (equivalently, $J(f)$ is connected), or $\exp(\sL_f)\notin\overline{\Q}$. See Corollary~\ref{LyaAlg}.

\subsection{An analogue for critical heights}
Let $f\in\overline{\Q}[z]$ be a polynomial of degree $d\geq2$. The \emph{canonical height function} of $f$ is the function $\hat{h}_f:\overline{\Q}\to\R_{\geq0}$ given by
$$\hat{h}_f(a)=\lim\limits_{n\to\infty}\frac{1}{d^n}h(f^{\circ n}(a)), a\in\overline{\Q},$$
where $h:\overline{\Q}\to\R_{\geq0}$ is the (absolute logarithmic) height function on $\overline{\Q}$. The zero locus $\hat{h}_f^{-1}(0)$ equals $\PrePer(f,\overline{\Q})$, the set of $f$-preperiodic points in $\overline{\Q}$. The \emph{multiplicative canonical height function} of $f$ is the function
$$\hat{H}_f=\exp(\hat{h}_f):\overline{\Q}\to\R_{\geq1}.$$
See \cite{CallSilverman} and \cite[Chapter~3]{Silverman2007} for more details about canonical heights.

The \emph{critical (canonical) height} of $f$ is defined by
\begin{equation}\label{crithtdef}
	\hat{h}_{\crit}(f):=\sum_{\{c\in\C:f^\prime(c)=0\}}\hat{h}_f(c)\geq0,
\end{equation}
where the critical points are counted with multiplicity; see \cite[\S~6.2]{Silverman2012}. The function is invariant under conjugacy, so it induces a function $\hat{h}_{\crit}:\MPoly^d(\C)\to\R_{\geq0}$ on the moduli space of degree $d$ polynomials $\MPoly^d(\C)$ (see \S~\ref{secmul}). Note that $\hat{h}_{\crit}(f)=0$ if and only if $f$ is post-critically finite, i.e., all critical points of $f$ are $f$-preperiodic. Ingram \cite{Ing12} proved that $\hat{h}_{\crit}$ is comparable to an ample Weil height on $\MPoly^d(\C)$. The \emph{multiplicative critical (canonical) height} of $f$ is
\begin{equation}\label{mcrithtdef}
	\hat{H}_{\crit}(f):=\exp\left(\hat{h}_{\crit}(f)\right)=\prod_{\{c\in\C:f^\prime(c)=0\}}\hat{H}_f(c)\geq1.
\end{equation}

The (multiplicative) canonical heights admit decompositions into local canonical heights \cite[Theorem~3.29]{Silverman2007}. The (multiplicative) critical height is a ``global'' invariant of the system $(\A^1,f)$. Considering all archimedean places, we say that $f$ has a disconnected Julia set after Galois conjugation if $J(\sigma(f))\subset\C$ is disconnected for some $\sigma\in\Gal(\overline{\Q}/\Q)$, where $\sigma(f)\in\overline{\Q}[z]$ is the polynomial obtained from $f$ by applying $\sigma$ to all coefficients and viewed as a polynomial over $\C$ via the fixed embedding $\overline{\Q}\hookrightarrow\C$. We can prove that $f$ has a disconnected Julia set after Galois conjugation if and only if $\hat{H}_{\crit}(f)\notin\overline{\Q}$ (Corollary~\ref{disJGal}).

Our main theorem, Theorem~\ref{thmmain}, deals with only one archimedean place, corresponding to the fixed embedding $\overline{\Q}\hookrightarrow\C$. The following theorem is a ``global'' analogue of Theorem~\ref{thmmain} considering all archimedean places; its proof requires \cite[Proposition~1.14]{trans}, a result on the transcendence of products of canonical heights for polynomials:
\begin{Thm}\label{thmht}
	Let $f,g\in\overline{\Q}[z]$ be polynomials of degree $d\geq2$. Assume that $f$ has a disconnected Julia set after Galois conjugation. Then $\hat{h}_{\crit}(f)=\hat{h}_{\crit}(g)$ if and only if $f$ and $\sigma(g)$ are intertwined for some $\sigma\in\Gal(\overline{\Q}/\Q)$.
\end{Thm}
We give the proof of Theorem~\ref{thmht} in \S~\ref{secpf}, which is almost the same as the proof of Theorem~\ref{thmmain}.

\subsection{Application to multiplier spectra}\label{secmul}
Let $d$ be an integer at least two. Let $\Poly^d$ be the irreducible affine (algebraic) variety of degree $d$ polynomials, which is isomorphic to $(\A^1\setminus\{0\})\times\A^d$. The geometric quotient $\MPoly^d=\Poly^d/ \Aff$ is the moduli space of degree $d$ polynomials, where $\Aff$ is the group of degree one polynomials acting on $\Poly^d$ by conjugation. The moduli space $\MPoly^d$ is an irreducible affine variety defined over $\Q$. See \cite[\S~1]{JX23} and \cite[\S~2.1]{VH2412}. The Lyapunov exponent is invariant under (affine) conjugacy \cite[Remark~6.46]{Silverman2012}. Then we obtain a map
$$\Lyap_d:\MPoly^d(\C)\to[\log(d),\infty),[f]\mapsto\sL_f,$$
which is continuous with respect to the complex topology on $\MPoly^d(\C)$ (cf. \cite[Theorem~86]{Berteloot} and \cite[Theorem~B]{Mane88}). The map $\Lyap_d$ is surjective, i.e., $\Lyap_d(\MPoly^d(\C))=[\log(d),\infty)$ (see Remark~\ref{rmkcond}).

Let $f\in\C[z]$ be a polynomial of degree $d$. For an $f$-periodic point $z_0\in\P^1(\C)$ with exact period $n\in\Z_{>0}$, the \emph{multiplier} of $f$ at $z_0$ is defined as the differential
$$\rho_f(z_0):=df^{\circ n}(z_0)\in\C,$$
and the \emph{length} of $f$ at $z_0$ is the modulus $\left|\rho_f(z_0)\right|\in\R_{\geq0}$. Note that the multiplier of $f$ at the superattracting fixed point $\infty$ is $\rho_f(\infty)=0$.

For $n\in\Z_{>0}$, let $S_n(f)$ (resp. $L_n(f)$) be the multiset of multipliers (resp. lengths) of $n$-periodic points of $f$ in $\P^1(\C)$, whose cardinality equals $d^n+1$. The \emph{multiplier spectrum} (resp. \emph{length spectrum}) of $f$ is the sequence
$$S(f)=(S_n(f))_{n=1}^\infty\text{ (resp. }L(f)=(L_n(f))_{n=1}^\infty\text{)}$$
of multisets. The length spectrum $L(f)$ (hence $S(f)$) can determine the Lyapunov exponent $\sL_f$. Note that $S_n(f)$, $S(f)$, $L_n(f)$, and $L(f)$ depend only on the conjugacy class of $f$; hence they induce maps on $\MPoly^d(\C)$, still denoted by $S_n$, $S$, $L_n$, and $L$. The first- and second-named authors, Ji and Xie, constructed the \emph{multiplier spectrum morphism}
$$\tilde{\tau}_d:\MPoly^d\to\A^M$$
on $\MPoly^d$, where $M=M(d)\geq1$ is an integer depending only on $d$. The map $\tilde{\tau}_d$ is a morphism between algebraic varieties (defined over $\Q$) such that $\tilde{\tau}_d([f])=\tilde{\tau}_d([g])$ if and only if $S([f])=S([g])$, for all $[f],[g]\in\MPoly^d(\C)$. See \cite{JX23} for more details. The morphism $\tilde{\tau}_d$ is generically injective:
\begin{Thm}[{\cite[Theorem~1.4]{JX23}}]\label{thmgeninj}
	For every integer $d\geq2$, the morphism $\tilde{\tau}_d$ is generically injective: there exist a non-empty Zariski open subset $U$ of $\MPoly^d(\C)$ and a Zariski open subset $W$ of the Zariski closure of $\tilde{\tau}_d(U)$ with $\tilde{\tau}_d^{-1}(W)=U$ such that $\tilde{\tau}_d:U\to W$ is a finite \'etale morphism of degree $1$.
\end{Thm}
Huguin \cite{VH2412} proved Theorem~\ref{thmgeninj} using completely different methods. The third-named author, Zhang, provided a detailed proof of Theorem~\ref{thmgeninj} in \cite{PolyMulti}, based on the ideas of Ji-Xie \cite{JX23} and Pakovich \cite{Paksimple,Pakdeg23}.

As an application of our main results, in \S~\ref{secpf} we show that Theorem~\ref{thmmain} implies Theorem~\ref{thmgeninj}. Thus, we obtain a new proof of Theorem~\ref{thmgeninj}.

We remark that the above constructions work for rational maps $f\in\C(z)$ of degree $d$. The analogous result \cite[Theorem~1.3]{JX23} of Theorem~\ref{thmgeninj} for rational maps is a generalization of a remarkable theorem of McMullen \cite{McMullen1987}.

\section{Preliminaries}
In \S~\ref{secLya}, we recall basic properties of Lyapunov exponents and Green functions, and deal with the ``easy'' direction of our main theorems. In \S~\ref{secBott}, we briefly recall the notion of B{\"o}ttcher coordinates for polynomials. In \S~\ref{secintw}, an equivalent definition of intertwined relation for non-exceptional polynomials is considered.

\subsection{Lyapunov exponents and Green functions}\label{secLya}
Let $f(z)=a_dz^d+\cdots+a_1z+a_0\in\C[z]$ be a polynomial of degree $d\geq2$.

The \emph{Green function} (also called the \emph{escape rate function}) of $f$ on $\C$ is the continuous subharmonic function $g_f:\C\to\R_{\geq0}$ defined by
$$g_f(z)=\lim\limits_{n\to\infty}\frac{1}{d^n}\log^+\left|f^{\circ n}(z)\right|,$$
where $\log^+=\max\{0,\log\}$. It satisfies $g_f(f(z))=dg_f(z)$ for $z\in\C$. The zero locus $g_f^{-1}(0)$ equals the \emph{filled Julia set}
$$K(f)=\{z\in\C:(f^{\circ n}(z))_{n\geq1}\text{ is bounded}\}$$
of $f$, which is compact in $\C$. Then the Lyapunov exponent $\sL_f$ of $f$ can be computed by Przytycki's formula:
\begin{equation}\label{LyapGreen}
	\sL_f=\log(d)+\sum_{\{c\in\C:f^\prime(c)=0\}}g_f(c),
\end{equation}
where the critical points are counted with multiplicity; see \cite{Prz85,Man84,Mane88}. By \eqref{LyapGreen} and \cite[Theorem~9.5]{Milnor}, $\sL_f\geq\log(d)$ and the following statements are equivalent:
\begin{itemize}
	\item $\sL_f=\log(d)$;
	\item all finite critical points of $f$ belong to the filled Julia set $K(f)$;
	\item $K(f)$ is connected;
	\item $J(f)$ is connected.
\end{itemize}

The polynomial $f$ is called \emph{exceptional} if it is (affinely) conjugate to $z^d$ or $\pm T_d(z)$, where $T_d(z)$ is the (normalized) Chebyshev polynomial of degree $d$, i.e., the unique monic polynomial of degree $d$ satisfying $T_d(z+z^{-1})=z^d+z^{-d}$. See \cite[Definition~1.1]{PolyMulti}. If $f$ is exceptional, then its Lyapunov exponent $\sL_f$ attains the minimum $\sL_f=\log(d)$ (cf. \cite[\S~2]{zdunik2014characteristic}).

The ``if'' direction of Theorem~\ref{thmmain} and Theorem~\ref{thmht} follows from the following two easy lemmas:
\begin{Lem}\label{intwLyapsame}
	Let $f,g\in\C[z]$ be polynomials of degree $d\geq2$. If $f$ and $\hat{g}$ are intertwined for some $\hat{g}\in\{g,\overline{g}\}$, then $\sL_f=\sL_g$.
\end{Lem}
\begin{proof}
	By \eqref{LyapGreen} and the definition of Green functions, it is easy to see that $\sL_g=\sL_{\overline{g}}$. Hence we may assume that $f$ and $g$ are intertwined, and the conclusion follows from \cite[Lemma~3.48]{FG22}.
\end{proof}
\begin{Lem}\label{intwCrithtsame}
	Let $f,g\in\overline{\Q}[z]$ be polynomials of degree $d\geq2$. If $f$ and $\sigma(g)$ are intertwined for some $\sigma\in\Gal(\overline{\Q}/\Q)$, then $\hat{h}_\crit(f)=\hat{h}_\crit(g)$.
\end{Lem}
\begin{proof}
	By the definition of $\hat{h}_\crit$ and the Galois invariance of the (absolute logarithmic) height $h$, it is easy to see that $\hat{h}_\crit(g)=\hat{h}_\crit(\sigma(g))$ for all $\sigma\in\Gal(\overline{\Q}/\Q)$. Hence we may assume that $f$ and $g$ are intertwined, and the conclusion follows from the analogue of \eqref{LyapGreen} over any metrized field of characteristic zero (Okuyama \cite[\S~5]{Okuyama}), \cite[Lemma~3.48]{FG22}, and decompositions of canonical heights into local canonical heights \cite[Theorem~3.29]{Silverman2007}.
\end{proof}

\subsection{B{\"o}ttcher coordinates}\label{secBott}
A \emph{B{\"o}ttcher coordinate} of $f$ (at infinity) is a Laurent series $\phi_f(z)\in\C((z^{-1}))$ satisfying
$$(\phi_f\circ f)(z)=(z^d\circ\phi_f)(z)$$
and having a simple pole at infinity, i.e., it takes the form
$$\phi_f(z)=b_1z+b_0+b_{-1}z^{-1}+b_{-2}z^{-2}+\cdots.$$
It exists and is unique up to multiplication by a $(d-1)$-th root of unity in $\C$ (see \cite[Proposition~4.1 and Lemma~4.4]{trans}). From now on, we use $\phi_f$ to denote a fixed B{\"o}ttcher coordinate of $f$. There is a constant $B(f)>0$ such that $\phi_f$ converges in the punctured neighborhood of infinity
$$\Sigma(f):=\{z\in\C:\left|z\right|>B(f)\}$$
and $f(\Sigma(f))\subseteq\Sigma(f)$ (cf. \cite[Proposition~4.3]{trans}). We may assume that $\Sigma(f)\cap K(f)=\emptyset$ and $B(f)=B(\overline{f})$. Then $\Sigma(f)=\Sigma(\overline{f})=\overline{\Sigma(f)}$. By \cite[Proposition~2.22]{FG22} (after enlarging $B(f)$ if necessary), the B{\"o}ttcher coordinate and Green function of $f$ are related by
\begin{equation}\label{BottGreen}
	g_f(z)=\log\left|\phi_f(z)\right|,\ z\in\Sigma(f).
\end{equation}
Note that $\phi_f(z)\neq0$ for every $z\in\Sigma(f)$. See \cite[\S~2.5]{FG22} and \cite[\S~4.1]{trans} for a detailed treatment of B{\"o}ttcher coordinates.

\subsection{The intertwined relation for non-exceptional polynomials}\label{secintw}
The article \cite{trans} involves an equivalence relation $\sim$ between rational maps, whose definition is slightly different from Definition \ref{intwdef} for the polynomial case:
\begin{Def}[{\cite[\S~3.1]{trans}}]
	Let $f,g\in\C(z)$ be rational maps of degree at least $2$. We say that $f$ and $g$ are \emph{semi-equivalent} and write $f\sim g$ if there exists an irreducible algebraic curve $Z\subset\P^1\times\P^1$ over $\C$ whose projections to both factors $\P^1$ are finite, such that $Z$ is periodic under the endomorphism $$(f,g):\P^1\times\P^1\to\P^1\times\P^1,(x,y)\mapsto(f(x),g(y)).$$
\end{Def}
\begin{Rem}\label{rmksim}
	Let $f,g\in\C(z)$ be rational maps of degree $\geq2$. If $f\sim g$, then $\deg(f)=\deg(g)$ (see \cite[Remark~3.1]{trans}). By \cite[\S~3.1]{trans} and the Riemann-Hurwitz formula, $f\sim g$ if and only if there exist non-constant rational maps $R,h_1,h_2\in\C(z)\setminus\C$ and $n\in\Z_{>0}$ such that
	\begin{equation}
		f^{\circ n}\circ h_1=h_1\circ R\quad\text{and}\quad g^{\circ n}\circ h_2=h_2\circ R.
	\end{equation}
\end{Rem}

We show that $\sim$ coincides with the intertwined relation for non-exceptional polynomials:
\begin{Lem}\label{simintw}
	Let $f,g\in\C[z]$ be non-exceptional polynomials of degree $\geq2$. Then $f\sim g$ if and only if $f$ and $g$ are intertwined.
\end{Lem}
\begin{proof}
	The ``if'' direction is trivial.
	
	It suffices to prove the ``only if'' direction. Let $f,g\in\C[z]$ be non-exceptional polynomials of degree $\geq2$ such that $f\sim g$. Set $d:=\deg(f)=\deg(g)\geq2$ (see Remark~\ref{rmksim}). Let $Z\subset\P^1\times\P^1$ be an irreducible complex algebraic curve whose projections to both factors $\P^1$ are finite, such that $Z$ is periodic under the endomorphism $(f,g):\P^1\times\P^1\to\P^1\times\P^1$. Take an integer $l\geq1$ such that $Z$ is $(f^{\circ l},g^{\circ l})$-invariant. Since $f$ and $g$ are non-exceptional, the iterates $f^{\circ l}$ and $g^{\circ l}$ are also non-exceptional (cf. \cite{Zdunik90}). We may assume that $l=1$, i.e., $Z$ is $(f,g)$-invariant.
	
	Let
	$$\nu:\hat{Z}\to Z$$
	be the normalization of the irreducible projective curve $Z$. The restriction of the endomorphism $(f,g)$ to $Z$ lifts to $\hat{Z}$ and defines a non-invertible endomorphism $F:\hat{Z}\to\hat{Z}$. Since $\deg(F)=\deg(f)=d\geq2$, by the Riemann-Hurwitz formula, $\hat{Z}$ has genus $g(\hat{Z})\leq1$.
	
	Assume first that $g(\hat{Z})=1$. Then $\hat{Z}$ is an elliptic curve. Since
	$$f\circ(\pi_1\circ\nu)=(\pi_1\circ\nu)\circ F,$$
	the polynomial $f$ is a Latt\`es map (cf. \cite{milnor2006lattes}), which is a contradiction because any Latt\`es map has Julia set $\P^1(\C)$.
	
	Assume that $g(\hat{Z})=0$. Then $\hat{Z}\cong\P^1$. We identify $\hat{Z}$ with $\P^1$, so that
	$F$, $h_1:=\pi_1\circ\nu$, and $h_2:=\pi_2\circ\nu$ are non-constant rational maps in $\C(z)\setminus\C$. Set
	$$C:=h_2^{-1}(\infty)\subset \hat{Z}(\C)=\P^1(\C).$$
	Then $C$ is a non-empty finite set because $h_2\in\C(z)\setminus\C$. Note that $C$ is totally $F$-invariant, i.e., $F^{-1}(C)=C=F(C)$. By \cite[Theorem~1.6]{Silverman2007}, we have $\#C\in\{1,2\}$.
	
	If $\#C=2$, then \cite[Theorem~1.6(b)]{Silverman2007} shows that $F$ is conjugate to $z^{\pm d}$ by a M\"obius transformation. It is well-known that $f$ (and $g$) must be exceptional polynomials (cf. \cite[Remark~1.5]{cyc} and \cite{milnor2006lattes}), a contradiction.
	
	Thus $\#C=1$. By \cite[Theorem~1.6(a)]{Silverman2007}, after conjugacy by M\"obius transformations, we may assume that $C=\{\infty\}$ and $F\in\C[z]$ is a polynomial. Since $h_2^{-1}(\infty)=C=\{\infty\}$, the map $h_2$ is also a polynomial. The polynomial $F$ is non-exceptional because $f$ is non-exceptional (cf. \cite[Remark~1.5]{cyc} and \cite{milnor2006lattes}). Let $C^\prime:=h_1^{-1}(\infty)\subset\P^1(\C)$. Then $C^\prime$ is a non-empty finite set which is totally $F$-invariant. Since $F$ is a non-exceptional polynomial, by \cite[Theorem~1.6]{Silverman2007} we conclude that $h_1^{-1}(\infty)=C^\prime=\{\infty\}$. Thus $h_1$ is also a polynomial.  Therefore, $f$ and $g$ are intertwined by \eqref{intwdef2}.
\end{proof}
\begin{Rem}
	However, Lemma~\ref{simintw} is false for exceptional polynomials. Let $d\geq2$ be an integer. By
	$$T_d\circ\left(z+\frac{1}{z}\right)=\left(z+\frac{1}{z}\right)\circ z^d,$$
	we see that $T_d(z)\sim z^d$. On the other hand, it is clear that $T_d(z)$ and $z^d$ are not conjugate, hence they are not intertwined by \cite[Theorem~3.39]{FG22}.
\end{Rem}

\section{Algebraicity of B{\"o}ttcher coordinates and multiplicative canonical heights}\label{secalg}
The second-named author studied the transcendence of products of B{\"o}ttcher coordinates in \cite{trans}. The following theorem is an easy consequence of \cite[Theorem~1.13]{trans}:
\begin{Thm}\label{transBott}
	Let $r\in\Z_{>0}$ and $f_1,\dots,f_r\in\overline{\Q}[z]$ be non-exceptional polynomials of degree $d\geq2$. For $1\leq i\leq r$, let $a_i\in\overline{\Q}\cap\Sigma(f_i)\subseteq\C\setminus K(f_i)$ and let $n_i\in\Z\setminus\{0\}$. Assume that
	$$\prod_{i=1}^r\phi_{f_i}(a_i)^{n_i}\in\overline{\Q}$$
	is algebraic. Then $r\geq2$ and there exist indices $1\leq i_+\neq i_-\leq r$ with $n_{i_+}\geq1$ and $n_{i_-}\leq-1$ such that $f_{i_+}$ and $f_{i_-}$ are intertwined.
\end{Thm}
\begin{proof}
	Fix a number field $K$ such that $f_i\in K[z]$ and $a_i\in K$ for $1\leq i\leq r$. We work with the archimedean place $v_0$ of $K$ corresponding to the fixed embedding $K\hookrightarrow\overline{\Q}\hookrightarrow\C$.
	
	By \cite[Theorem~1.13]{trans}, $\prod_{i=1}^r\phi_{f_i}(a_i)^{n_i}$ is a root of unity and we have
	\begin{equation}\label{partieq}
		\sum_{s\in J_j}n_s d_{s/j}=0
	\end{equation}
	for $1\leq j\leq l$. Here $l\in\Z_{>0}$, $\{1,\dots,r\}=\sqcup_{j=1}^l J_j$ is a partition of $\{1,\dots,r\}$ ($J_j\neq\emptyset$ for $1\leq j\leq l$), and $d_{s/j}$ is a positive integer for $1\leq j\leq l$ and $s\in J_j$. These data
	$$(\{1,\dots,r\}=\sqcup_{j=1}^l J_j,(d_{s/j})_{1\leq j\leq l,s\in J_j})$$
	are purely geometric invariants of $(f_1,a_1),\dots,(f_r,a_r)$; see \cite[\S~4.5]{trans} for details.
	
	Since every $d_{s/j}$ ($1\leq j\leq l$, $s\in J_j$) is positive and every $n_s$ ($1\leq s\leq r$) is nonzero, from \eqref{partieq} we conclude that for every $1\leq j\leq l$, there exist $s_{j,+}, s_{j,-}\in J_j$ such that
	$$n_{s_{j,+}}\geq1\quad\text{and}\quad n_{s_{j,-}}\leq-1;$$
	so $\#J_j\geq2$, and therefore $r\geq2$. Set $i_+=s_{1,+}$ and $i_-=s_{1,-}$. Then $n_{i_+}\geq1$, $n_{i_-}\leq-1$, and $i_+,i_-$ are distinct indices in $J_1$. By Lemma~\ref{simintw} and \cite[Remarks~4.12~and~4.16]{trans}, $f_i$ and $f_{i^\prime}$ are intertwined if $i$ and $i^\prime$ belong to the same $J_j$ ($1\leq j\leq l$). Hence $f_{i_+}$ and $f_{i_-}$ are intertwined. This completes the proof.
\end{proof}
\begin{Rem}\label{transBottrmk}
The product $\prod_{i=1}^r\phi_{f_i}(a_i)^{n_i}$ must be a root of unity if it is algebraic \cite[Theorem~1.13]{trans}. In our setting, we only need to consider the case where $\prod_{i=1}^r\phi_{f_i}(a_i)^{n_i}$ is a $(d-1)$-th root of unity.
\end{Rem}
As a corollary, we determine all possible algebraic values of $\exp\left(\sL_f\right)$ where $f\in\C[z]$ is a polynomial of degree $d$:
\begin{Cor}\label{LyaAlg}
	Let $f\in\overline{\Q}[z]$ be a polynomial of degree $d\geq2$. Then either $\exp\left(\sL_f\right)\notin\overline{\Q}$, or $\sL_f=\log(d)$ (and $\exp\left(\sL_f\right)=d$).
\end{Cor}
\begin{proof}
Suppose, for the sake of contradiction, that $\exp\left(\sL_f\right)\in\overline{\Q}\setminus\{d\}$. In particular, $f$ is non-exceptional. Since $\sL_f\geq\log(d)$, by \eqref{LyapGreen}, we obtain
\begin{equation}\label{Lyacor1}
	\prod_{\{c\in\C\colon f^\prime(c)=0\}}\exp(g_f(c))\in\overline{\Q}\cap\R_{>1}.
\end{equation}
Let $c_1,\dots,c_{d-1}$ be the critical points of $f$ in $\C$ (in fact, in $\overline{\Q}$) counted with multiplicity. By \eqref{Lyacor1}, after renumbering, we may assume there is an integer $r\in\{1,\dots,d-1\}$ such that for $1\leq i\leq d-1$,
$$g_f(c_i)>0\text{ (equivalently, }c_i\notin K(f)\text{)}\quad\text{if and only if}\quad1\leq i\leq r.$$
Fix an integer $l\geq 1$ such that $a_i:=f^{\circ l}(c_i)\in\Sigma(f)=\Sigma(\overline{f})$ for all $1\leq i\leq r$. Taking the $d^l$-th power of \eqref{Lyacor1}, by \eqref{BottGreen} we get
\begin{equation}\label{Lyacor2}
	\left| \prod_{i=1}^r \phi_f(a_i)\right|\in\overline{\Q}\cap\R_{>1}.
\end{equation}
Let $\overline{\phi_f}$ be the Laurent series obtained from $\phi_f$ by applying complex conjugation to all coefficients. It is clear that $\overline{\phi_f}$ is a B{\"o}ttcher coordinate of $\overline{f}$; so
\begin{equation}\label{zeta}
	\overline{\phi_f}=\zeta\phi_{\overline{f}}
\end{equation}
for some $(d-1)$-th root of unity $\zeta\in\C^\times$. Taking the square of \eqref{Lyacor2}, by \eqref{zeta} we see that
\begin{equation}\label{Lyacor3}
	\prod_{i=1}^r\left(\phi_f(a_i)\phi_{\overline{f}}(\overline{a_i})\right)\in\overline{\Q}.
\end{equation}
Applying Theorem~\ref{transBott} to the $2r$ pairs $(f,a_1),\dots,(f,a_r),(\overline{f},\overline{a_1}),\dots,(\overline{f},\overline{a_r})$ and integers $n_1=\cdots=n_{2r}=1$, we conclude that there exists an index $1\leq i_-\leq 2r$ such that $1=n_{i_-}\leq-1$, a contradiction.
\end{proof}

Let $d\geq2$ be an integer. Set
$$\sD_d:=\{f\in\overline{\Q}[z]\colon f\text{ is non-exceptional of degree }d\}$$
and
\begin{equation*}
	\sT_d:=\{(f,a)\in\sD_d\times\overline{\Q}\colon\sigma(a)\notin K(\sigma(f))\text{ for some }\sigma:\overline{\Q}\hookrightarrow\C\}.
\end{equation*}
For multiplicative canonical heights, we obtain the following result from \cite[Proposition~1.14]{trans} similarly:
\begin{Prop}\label{transht}
	Let $d\geq2$ and $r\geq1$ be integers. Let $(f_1,a_1),\dots,(f_r,a_r)$ be $r$ pairs in $\sT_d$. For $1\leq i\leq r$, let $n_i\in\Z\setminus\{0\}$. Assume that
	$$\prod_{i=1}^r\hat{H}_{f_i}(a_i)^{n_i}\in\overline{\Q}$$
	is algebraic. Then $r\geq2$ and there exist indices $1\leq i_+\neq i_-\leq r$ with $n_{i_+}\geq1$ and $n_{i_-}\leq-1$ such that $f_{i_+}$ and $f_{i_-}$ are intertwined.
\end{Prop}
\begin{proof}
	For $(f,a),(g,b)\in\sT_d$, we write $(f,a)\sim_w(g,b)$ if
	$$\overline{O_{(f,\sigma(g))}((a,\sigma(b)))}^{\Zar}\neq\A^2$$
	for some $\sigma\in\Gal(\overline{\Q}/\Q)$. Then $\sim_w$ defines an equivalence relation on $\sT_d$ \cite[\S~4.6]{trans}. By Lemma~\ref{simintw} and \cite[Remark~4.12]{trans}, if $(f,a)\sim_w(g,b)$ in $\sT_d$, then $f$ and $\sigma(g)$ are intertwined for some $\sigma\in\Gal(\overline{\Q}/\Q)$.
	
	Let $\{1,\dots,r\}=\sqcup_{j=1}^l J_j$ ($l\in\Z_{>0}$) be a partition such that $J_j\neq\emptyset$ for $1\leq j\leq l$ and that for all $1\leq i_1,i_2\leq r$, $(f_{i_1},a_{i_1})\sim_w(f_{i_2},a_{i_2})$ if and only if $i_1,i_2\in J_j$ for some $1\leq j\leq l$. For every $1\leq j\leq l$, we can associate a rational point $q_j:=Q((f_s,a_s),s\in J_j)\in\P^{\#J_j-1}(\Q)$, which is a geometric invariant up to Galois conjugation \cite[\S~4.6]{trans}. Denote the homogeneous coordinates of $\P^{\#J_j-1}$ by $[z_s]$ ($s\in J_j$). By \cite[\S~4.6 and \S~4.4]{trans}, the point $q_j$ admits a homogeneous coordinate representation of the form $[m_{s/j}]_{s\in J_j}$, where $m_{s/j}\in\Z_{>0}$ for every $s\in J_j$.
	
	Since $\prod_{i=1}^r\hat{H}_{f_i}(a_i)^{n_i}$ is algebraic, by \cite[Proposition~1.14]{trans}, we have
	\begin{equation}\label{partieq2}
		\sum_{s\in J_j}n_s m_{s/j}=0
	\end{equation}
	for $1\leq j\leq l$. Since every $m_{s/j}$ ($1\leq j\leq l$, $s\in J_j$) is positive and every $n_s$ ($1\leq s\leq r$) is nonzero, from \eqref{partieq2} we conclude that for every $1\leq j\leq l$, there exist $s_{j,+}, s_{j,-}\in J_j$ such that
	$$n_{s_{j,+}}\geq1\quad\text{and}\quad n_{s_{j,-}}\leq-1;$$
	so $\#J_j\geq2$, and therefore $r\geq2$. Set $i_+=s_{1,+}$ and $i_-=s_{1,-}$. Then $n_{i_+}\geq1$, $n_{i_-}\leq-1$, and $i_+,i_-$ are distinct indices in $J_1$. The polynomials $f_{i_+}$ and $f_{i_-}$ are intertwined by the construction of the partition. This completes the proof.
\end{proof}
\begin{Rem}
	By \cite[Remark~4.26]{trans}, $\prod_{i=1}^r\hat{H}_{f_i}(a_i)^{n_i}$ is algebraic if and only if $\prod_{i=1}^r\hat{H}_{f_i}(a_i)^{n_i}=1$. Since we need to consider Galois conjugation in Proposition~\ref{transht}, the partition of $\{1,\dots,r\}$ in Proposition~\ref{transht} is ``coarser'' (it may not be strictly coarser) than that in Theorem~\ref{transBott}.
\end{Rem}
Nguyen \cite[Corollary~1.6]{ngu23} (see also \cite[\S~4.6]{trans}) proved that for a non-exceptional polynomial $f\in\overline{\Q}[z]$ of degree $\geq2$ and an algebraic point $a\in\overline{\Q}$,
\begin{equation}\label{Ngu}
	\hat{H}_f(a)\in\overline{\Q}\quad\text{if and only if}\quad\sigma(a)\in K(\sigma(f))\text{ for every }\sigma\in\Gal({\overline{\Q}/\Q}).
\end{equation}
Combining \eqref{Ngu} and Proposition \ref{transht}, we obtain a characterization for $f$ to have a disconnected Julia set after Galois conjugation:
\begin{Cor}\label{disJGal}
	For a polynomial $f\in\overline{\Q}[z]$ of degree $d\geq2$, the following statements are equivalent:
	\begin{enumerate}
		\item $f$ has a disconnected Julia set after Galois conjugation;
		\item $\hat{H}_{\crit}(f)\notin\overline{\Q}$.
	\end{enumerate}
\end{Cor}
\begin{proof}
	If $f$ is exceptional, then $\sigma(f)$ is exceptional and $J(\sigma(f))$ is connected for every $\sigma\in\Gal({\overline{\Q}/\Q})$. Exceptional polynomials are post-critically finite, hence their multiplicative critical heights are equal to $1$. We may assume that $f$ is non-exceptional.
	
	Let $c_1,\dots,c_{d-1}$ be the critical points of $f$ in $\C$ (in fact, in $\overline{\Q}$) counted with multiplicity. By definition, we have $\hat{H}_{\crit}(f)=\prod_{i=1}^{d-1}\hat{H}_f(c_i)$. After renumbering, we may assume that there is a (non-negative) integer $r\in\{0,1,\dots,d-1\}$ such that for $1\leq i\leq d-1$,
	$$\sigma(c_i)\notin K(\sigma(f))\text{ for some }\sigma\in\Gal({\overline{\Q}/\Q})\quad\text{if and only if}\quad1\leq i\leq r.$$
	Then (1) is equivalent to: $(1)^\prime$ $r\geq1$. See \S~\ref{secLya}.
	
	By \eqref{Ngu}, for $r+1\leq i\leq d-1$, we have $\hat{H}_f(c_i)\in\overline{\Q}\cap\R_{\geq1}$. Therefore, (2) is equivalent to: $(2)^\prime$ $r\geq1$ and $\prod_{i=1}^{r}\hat{H}_f(c_i)\notin\overline{\Q}$.
	
	It suffices to show that $(1)^\prime\Rightarrow(2)^\prime$. Assume $r\geq1$. Suppose, for the sake of contradiction, that $\prod_{i=1}^{r}\hat{H}_f(c_i)\in\overline{\Q}$. Applying Proposition \ref{transht} to the pairs $(f,c_1),\dots,(f,c_r)\in\sT_d$ and integers $n_1=\cdots=n_r=1$, we conclude that there exists an index $1\leq i_-\leq r$ such that $1=n_{i_-}\leq-1$, a contradiction.
\end{proof}

\section{Proof of the main results}\label{secpf}
\begin{proof}[Proof of Theorem~\ref{thmmain}]
	The ``if'' direction follows from Lemma~\ref{intwLyapsame}. It suffices to prove the ``only if'' direction. Suppose that $f,g\in\overline{\Q}[z]$ are polynomials of degree $d\geq2$ such that
	\begin{equation}\label{Lyapequal1}
		\sL_f=\sL_g>\log(d).
	\end{equation}
	
	Let $c_1,\dots,c_{d-1}$ (resp. $\la_1,\dots,\la_{d-1}$) be the critical points of $f$ (resp. $g$) in $\C$ (in fact, in $\overline{\Q}$) counted with multiplicity. By \eqref{LyapGreen}, equation \eqref{Lyapequal1} becomes
	\begin{equation}\label{Lyapequal2}
		\sum_{i=1}^{d-1}g_f(c_i)=\sum_{i=1}^{d-1}g_g(\la_i)>0.
	\end{equation}
	After renumbering, we may assume that there are positive integers $1\leq s,t\leq d-1$ such that for $1\leq i\leq d-1$, $c_i\notin K(f)$ if and only if $1\leq i\leq s$, and  $\la_i\notin K(g)$ if and only if $1\leq i\leq t$. Then \eqref{Lyapequal2} becomes
	\begin{equation}\label{Lyapequal3}
		\sum_{i=1}^{s}g_f(c_i)=\sum_{i=1}^{t}g_g(\la_i),
	\end{equation}
	where all terms on both sides are strictly positive.
	
	Take a positive integer $N\geq1$ such that
	$$a_i:=f^{\circ N}(c_i)\in\Sigma(f)=\Sigma(\overline{f})$$
	for $1\leq i\leq s$, and
	$$b_i:=g^{\circ N}(\la_i)\in\Sigma(g)=\Sigma(\overline{g})$$
	for $1\leq i\leq t$. Multiplying both sides of \eqref{Lyapequal3} by $d^N$ and using \eqref{BottGreen}, we obtain
	\begin{equation}\label{Lyapequal4}
		\log\left|\prod_{i=1}^s\phi_f(a_i)\right|=\log\left|\prod_{i=1}^t\phi_g(b_i)\right|.
	\end{equation}
	
	Note that there are $(d-1)$-th roots of unity $\zeta_1,\zeta_2\in\C^\times$ such that
	\begin{equation}\label{zeta2}
		\overline{\phi_f}=\zeta_1\phi_{\overline{f}}\quad\text{and}\quad\overline{\phi_g}=\zeta_2\phi_{\overline{g}};
	\end{equation}
	see \eqref{zeta}. Then from \eqref{Lyapequal4} and \eqref{zeta2} we obtain
	\begin{equation}\label{Lyapequal5}
		\prod_{i=1}^s\left(\phi_f(a_i)\phi_{\overline{f}}(\overline{a_i})\right) \prod_{j=1}^t\left(\phi_g(b_j)^{-1}\phi_{\overline{g}}(\overline{b_j})^{-1}\right)=\zeta_1^{-s}\zeta_2^t,
	\end{equation}
	which is a $(d-1)$-th root of unity.
	
	Since $\sL_f=\sL_{\overline{f}}=\sL_g=\sL_{\overline{g}}>\log(d)$, the polynomials $f$, $\overline{f}$, $g$, and $\overline{g}$ are all non-exceptional. Applying Theorem~\ref{transBott} to \eqref{Lyapequal5}, we conclude that there exist $\tilde{f}\in\{f,\overline{f}\}$ and $\tilde{g}\in\{g,\overline{g}\}$ such that $\tilde{f}$ and $\tilde{g}$ are intertwined. By \eqref{intwdef2}, we see that either $f$ and $g$ are intertwined, or $f$ and $\overline{g}$ are intertwined. This completes the proof.
\end{proof}
\begin{proof}[Proof of Theorem~\ref{thmht}]
	By Lemma~\ref{intwCrithtsame}, it suffices to prove the ``only if'' direction. As $\hat{h}_{\crit}(f)=\hat{h}_{\crit}(g)$, by Corollary \ref{disJGal}, $g$ has a disconnected Julia set after Galois conjugation. With the help of \eqref{Ngu}, the proof is completed by replacing \eqref{LyapGreen} and Theorem~\ref{transBott} in the proof of Theorem~\ref{thmmain} with \eqref{crithtdef} and Proposition~\ref{transht}, respectively, and arguing similarly.
\end{proof}
\begin{proof}[Proof of Theorem~\ref{thmgeninj} using Theorem~\ref{thmmain}]
	Let $d\geq2$ be an integer.
	
	By \cite[Theorem~30]{Berteloot}, Lyapunov exponents are determined by length spectra (hence also by multiplier spectra); that is, for all $f,g\in\C[z]$ of degree $d$, if $L(f)=L(g)$, then $\sL_f=\sL_g$.
	
	Let
	$$\Man_d:=\{[f]\in\MPoly^d(\C):\sL_f=\log(d)\}$$
	be the locus of conjugacy classes of degree $d$ polynomials with connected Julia sets, which is compact and closed in the affine variety $\MPoly^d(\C)$ with respect to the complex topology (cf. \cite[Corolloary 3.7]{BH88}).
	
	By \cite[Theorems~3.51~and~3.52]{FG22} (see also \cite[Theorem~1.5 and Proposition~1.7]{PolyMulti}), there exists a non-empty Zariski affine open subset $U=U(d)\subseteq\MPoly^d(\C)$ satisfying the following properties:
	\begin{itemize}
		\item $U$ is invariant under complex conjugation, i.e., $\tau([f]):=[\overline{f}]\in U$ for all $[f]\in U$;
		\item for all $[f],[g]\in U$, $f$ and $g$ are intertwined if and only if $[f]=[g]$, i.e., $f$ and $g$ are (affinely) conjugate.
	\end{itemize}
	The complex conjugation $\tau$ on $\MPoly^d$ (over $\Spec(\C)$) is the semi-linear automorphism induced by the involution $\Spec(\C)\to\Spec(\C)$ corresponding to the usual complex conjugation $\C\to\C$; see \cite[Chapter~II, Ex.~4.7]{52}. Note that $\tau$ preserves $\overline{\Q}$-points. Let $W$ be a non-empty Zariski affine open subset of the Zariski closure of $\tilde{\tau}_d(U)$ such that $\tilde{\tau}_d^{-1}(W)\subseteq U$ and $\tilde{\tau}_d:\tilde{\tau}_d^{-1}(W)\to W$ is a finite \'etale morphism of degree $m\geq1$. (Observe that $m=m(d)$ is independent of the choice of $W$.) After shrinking $U$ and $W$ if necessary, we may assume $U=\tilde{\tau}_d^{-1}(W)$. Recall that $\MPoly^d$ is defined over $\Q$, so we could take $U$ to be defined over $\overline{\Q}$. From now on, we only consider conjugacy classes in $U$. Let $U(\overline{\Q}):=U\cap\MPoly^d(\overline{\Q})$ be the set of $\overline{\Q}$-points in $U$.
	
	Suppose, for the sake of contradiction, that $m\geq2$. As in the proof of \cite[Theorem~1.3]{JX23}, we can find two algebraic families $h_1,h_2:C\to\MPoly^d$ of degree $d$ polynomials parametrized by a geometrically integral affine curve $C$, defined over $\overline{\Q}$, satisfying:
	\begin{itemize}
		\item $h_i$ is non-isotrivial (i.e., $h_i(C)$ is not a single point in $\MPoly^d$), $1\leq i\leq2$;
		\item $h_i(C(\C))\subseteq U$, $1\leq i\leq2$;
		\item $h_1(t)\neq h_2(t)$, for every $t\in C(\C)$;
		\item $\tilde{\tau}_d\circ h_1=\tilde{\tau}_d\circ h_2\colon C\to\A^M$, where $M=M(d)\geq1$ is as in \S~\ref{secmul}.
	\end{itemize}
	
	We introduce several subsets of $U\times U$ (of complex points). Let
	$$\Gamma:=\{([f],[\overline{f}]):[f]\in U\}\subset U\times U$$
	be the graph of the complex conjugation map $\tau$ on $U$. Note that $\Gamma$ is complex closed in $U\times U$. Let
	$$F:=\{([f],[g])\in U\times U:\tilde{\tau}_d([f])=\tilde{\tau}_d([g])\}\subseteq U\times U$$
	be the locus of pairs of conjugacy classes with the same multiplier spectrum over $U$, which is Zariski closed in $U\times U$. Let $\Delta\subset U\times U$ be the diagonal.
	
	Let $E$ be the Zariski closure of $(h_1,h_2)(C(\C))$ in $U\times U$, which is an irreducible (algebraic) curve in $U\times U$ defined over $\overline{\Q}$. The curve $E$ is affine because $U\times U$ is affine. Note that we have
	$$E\subseteq F\quad\text{and}\quad E\nsubseteq\Delta$$
	by the construction of $h_1$ and $h_2$. Then $E\cap\Delta$ is finite. After shrinking $U$ and $W$ if necessary, we may assume that $E\cap\Delta=\emptyset$ and $E$ is smooth. Let $E(\overline{\Q}):=E\cap U(\overline{\Q})^2$ be the set of $\overline{\Q}$-points in $E$.
	
	As multiplier spectra determine Lyapunov exponents, we have
	$$E\subseteq F\subseteq\Man_d^2\cup(U\setminus\Man_d)^2.$$
	In the smooth irreducible affine (complex algebraic) curve $E$, since $E\cap\Man_d^2$ is a complex compact subset, $V:=E\setminus\Man_d^2$ is a non-empty complex open subset of $E$. Because $E$ is defined over $\overline{\Q}$, the set $V(\overline{\Q}):=E(\overline{\Q})\setminus\Man_d^2$ of $\overline{\Q}$-points in $V$ is dense in $V$ in the complex topology. By Theorem~\ref{thmmain}, the construction of $U$, and the fact that $E\cap\Delta=\emptyset$, we obtain
	\begin{equation}\label{Vbar}
		V(\overline{\Q})\subseteq\Gamma.
	\end{equation}
	Taking closure in complex topology on both sides of \eqref{Vbar} in $U\times U$, we obtain $V\subseteq\Gamma$. 
	
	The complex open subset $V\subseteq E$ with complex structure is a complex manifold of dimension one (it might not be connected). For $1\leq i\leq2$, let $\pi_i:U\times U\to U$ be the projection onto the $i$-th factor which is an algebraic morphism over $\C$, and set $E_i=\pi_i(E)$. Then each $E_i$ is an irreducible complex algebraic curve in $U$. For $1\leq i\leq2$, the set
	$$V_i:=\pi_i(V)=E_i\setminus\Man_d$$
	is a non-empty complex open subset of $E_i$. Because $V$ is contained in the graph $\Gamma$ of $\tau\vert_U$, the restrictions $\pi_1\vert_V:V\to V_1$ and $\pi_2\vert_V:V\to V_2$ are bijective and holomorphic. We see that $V\cong V_1\cong V_2$ as complex manifolds, and that the restriction $\tau\vert_{V_1}$ can be expressed as
	\begin{equation}\label{tau}
		\tau\vert_{V_1}=\pi_2\circ(\pi_1\vert_V)^{-1}:V_1\to V_2.
	\end{equation}
	Since $\pi_1$ and $\pi_2$ are morphisms between $\C$-schemes (on $\C$-points), \eqref{tau} implies that the anti-holomorphic map $\tau\vert_{V_1}$ must be holomorphic on the complex manifold $V_1$. However, a map on a complex manifold $M$ can be simultaneously holomorphic and anti-holomorphic only if $M$ has complex dimension zero. Since $\dim_\C(V_1)=1$, we obtain a contradiction. This contradiction shows that $m=1$, i.e., $\tilde{\tau}_d$ is generically injective.
\end{proof}

\subsection*{Acknowledgement}
The first-named author, Zhuchao Ji, is supported by National Key R\&D Program of China (No. 2025YFA1018300), NSFC Grant (No. 12401106), and ZPNSF grant (No. XHD24A0201). The second- and third-named authors, Junyi Xie and Geng-Rui Zhang, are supported by NSFC Grant (No. 12271007).

\end{document}